\documentclass[psamsfonts]{amsart}  
\usepackage{amssymb}

\usepackage[dvips]{graphicx}

\usepackage{graphicx}
\usepackage{amssymb}                       
\usepackage{amsmath}
\usepackage{color}
\usepackage{times}

\usepackage[unicode,bookmarks,colorlinks]{hyperref}
\hypersetup{
   linkcolor=brickred,
}

\definecolor{mahogany}{cmyk}{0, 0.77, 0.87, 0}
\definecolor{salmon}{cmyk}{0, 0.53, 0.38, 0}
\definecolor{melon}{cmyk}{0, 0.46, 0.50, 0}
\definecolor{yellowgreen}{cmyk}{0.44, 0, 0.74, 0}
\definecolor{brickred}{cmyk}{0, 0.89, 0.94, 0.28}
\definecolor{OliveGreen}{cmyk}{0.64, 0, 0.95, 0.40}
\definecolor{RawSienna}{cmyk}{0, 0.72, 1.0, 0.45}
\definecolor{ZurichRed}{rgb}{1, 0, 0} 

\usepackage{fancyhdr}
\usepackage{amsmath,amstext,amssymb,amsopn,amsthm}
\usepackage{amsmath,amssymb,amsthm}
\usepackage[mathscr]{eucal}

\DeclareMathOperator{\sgn}{sgn}

\newcommand{\Rd}{{\mathbb{R}^d}}
\newcommand{\RR}{{\mathbb{R}}}

\newcommand{\R}{\mathbb{R}}
\newcommand{\ind}{\mathbf{1}}

\newcommand{\E}{\mathbb{E}}
\newcommand{\Ex}{{\mathbb{E}_x}}

\newcommand{\px}{{\mathbb{P}_x}}
\newcommand{\p}{{\mathbb{P}}}

\newcommand{\CD}{\mathcal{D}}
\newcommand{\cc}{\mathcal{C}}

\newtheorem{theorem}{Theorem}[section]

\numberwithin{equation}{section}

\newtheorem{lemma}[theorem]{Lemma} 

\newtheorem{corollary}[theorem]{Corollary}

\newtheorem{example}{Example}

\reversemarginpar
\definecolor{kb}{rgb}{0.1,0.5,0.1}
\definecolor{rb}{rgb}{0.1,0.2, 0.7}
\definecolor{tl}{rgb}{0.7,0.1,0.2}

\begin{document}

\title{Hardy-Stein identities and square functions for semigroups}

\author{Rodrigo Ba\~nuelos}
\address{Department of Mathematics, Purdue University\\
150 N. University Street, West Lafayette, IN 47907-2067, USA}
\email{banuelos@math.purdue.edu}

\author{Krzysztof Bogdan}
\address{Department of Mathematics, Wroc{\l}aw University of Technology, Wyb. Wyspia\'nskiego 27, 50-370 Wroc\l{}aw, Poland} 
\email{Krzysztof.Bogdan@pwr.edu.pl}
\thanks{R. Ba\~nuelos was supported in part  by NSF grant
\#1403417-DMS, K. Bogdan was supported in part by NCN grant \#2012/07/B/ST1/03356, T. Luks was supported in part by Agence Nationale de la Recherche grant ANR-09-BLAN-0084-01.}

\author{Tomasz Luks}
\address{Ecole Centrale de Marseille, I2M\\
38 Rue Fr\'ed\'eric Joliot Curie, 13013 Marseille, France
}
\email{tomasz.luks@centrale-marseille.fr}
\date{\today}
\maketitle
\begin{abstract}
We prove a Hardy-Stein type identity for the semigroups of symmetric, pure-jump L\'evy processes. 
Combined  with the Burkholder-Gundy inequalities, it gives the $L^p$ two-way boundedness, for $1<p<\infty$, of the corresponding Littlewood-Paley square function. The square function yields a direct proof of the $L^p$ boundedness of  Fourier multipliers obtained by transforms of martingales of L\'evy processes.
\end{abstract}

\section{Introduction}\label{sec:intro}
Littlewood and Paley introduced the square functions 
to harmonic analysis in \cite{MR1574750}.
Many applications and intrinsic beauty of the subject brought about 
enormous 
literature, which would be impossible to review here in a reasonably complete way.   
For results on classical square functions
we refer the reader to Zygmund \cite{Zyg} and Stein \cite{MR0290095}, \cite{MR0252961}. In particular, \cite{MR0290095} uses harmonic functions on the upper half-space and the related Gaussian and Poisson semigroups to develop  Littlewood-Paley theory for the $L^p$ spaces.   In \cite{MR0252961} Stein employs more general symmetric semigroups in a similar manner. He 
uses square functions defined in terms of the generalized Poisson semigroup, that is the original semigroup subordinated in the sense of Bochner by the $1/2$-stable subordinator \cite{MR2978140}.
He also proposes  square functions defined in terms of time derivatives of the original semigroup.
Similarly, Meyer \cite{Mey1} employs the generalized Poisson semigroup,
and Varopoulos in \cite{MR583240} uses time derivatives of the original semigroup.

It may be helpful to note that Littlewood-Paley theory and square functions (including the Lusin area integral) are 
auxiliary for studying $L^p$ and other function spaces, Fourier multipliers theorems, partial differential equations and boundary behavior of functions.  
This explains, in part, the large variety of square functions used in literature toward different goals. At the same time the multipliers and PDEs manageable by a square function depend on the semigroup employed in its definition, which motivates the study of square functions specifically related to a given semigroup. We also note that  square functions usually combine the {\it carr\'e du champ} corresponding to the semigroup \cite{Mey1} and integration against the semigroup or its Poisson subordination.

It is well-known that the probabilistic counterpart of square functions is the quadratic variation of 
the 
martingales. Similarly, the Littlewood-Paley inequalities for square functions may be considered analytic analogues of
the Burkholder-Davis-Gundy inequalities, which relate the $L^p$ integrability of the martingale and its maximal function to the $L^p$ integrability of its quadratic variation.
The probabilistic connections to Littlewood-Paley theory have been explored by countless  authors
for many years.  For a highly incomplete list of results, we refer the reader to  Stein \cite{MR0252961},  Meyer  \cite{Mey1}, \cite{Mey2}, \cite{Mey3}, Varopoulos \cite{MR583240}, Ba\~nuelos \cite{MR855179}, Ba\~nuelos and Moore \cite{MR1707297}, Bennett \cite{MR797052},  Bouleau and Lamberton \cite{BouLam}, Karli \cite{Kar1}, Kim and Kim \cite{MR2869738},  Krylov \cite{MR1317805}, and the many references given in these papers.  

In the analytic, as opposed to probabilistic, realm the $L^p$ boundedness of the classical Littlewood-Paley square functions
can be obtained from the Calder\'on-Zygmund theory of singular integrals, as done in Stein  
\cite[pp. 82-84]{MR0290095}. The singular integral approach can also be used for a wide range of Littlewood-Paley square functions constructed from  volume preserving dilations of approximations to the identity.  For this (well-known) approach, we refer the reader to \cite{MR1707297}.   An alternative beautiful way to prove $L^p$ boundedness in the classical case for $1<p< 2$ is via the so called Hardy-Stein identities.  This approach is employed in Stein \cite[pp. 86-88]{MR0290095} and, outside of some standard maximal function estimates that hold in very general settings when the Hardy-Littlewood maximal function is replaced by the semigroup maximal function, it is based on the fact that the Laplacian satisfies a special case of what in diffusion theory is often called the chain rule.  That is, $\Delta u^p ={p(p-1)}u^{p-2}|\nabla u|^2+pu^{p-1}\Delta u$ for $1<p<\infty$ and suitable functions $u$;  see \cite[Lemma 1, p. 86]{MR0290095}.  Stein's proof can be easily adapted to Markovian semigroups whose generators satisfy the chain rule as  discussed in \cite{MR2213477}, Formula (10).   
It is also explained in \cite{MR2213477} 
that such  chain rule requires the process to have continuous trajectories, 
thus ruling out the nonlocal operators.

The purpose of the present paper is to prove the two-way $L^p$ bounds for square functions of Markovian semigroups generated by nonlocal operators.
Indeed, we define an intrinsic square function $\tilde G(f)$ for such semigroups and prove the upper and lower boundedness in $L^p$.
The square function thus characterizes the $L^p$ spaces for 
$1<p<\infty$. 
We like to note a certain asymmetry in the definition of $\tilde G(f)$ and the fact that the more natural and symmetric square
function $G(f)$ fails to be bounded in $L^p$ for $1<p<2$. 

Our technique is based on new Hardy-Stein identities for the considered semigroups (which replace the chain rule for $1<p\leq 2$)  and on  Burkholder-Gundy inequalities for suitable martingales driven by the stochastic processes corresponding to those semigroups ({these are important} for $2\leq p<\infty$). Once the upper bound inequalities are obtained, the lower bound inequalities 
may be proved  by polarization and duality. 
Our Hardy-Stein identities are inspired by those given in \cite{MR3251822} for harmonic and conditionally harmonic functions of the Laplacian and the fractional Laplacian, but the present setting 
is distinctively different.

The paper may be considered as a streamlined approach from semigroups to Hardy-Stain identities to square functions to multiplier theorems.
To avoid certain 
technical problems our present results are restricted to
the (convolution) semigroups of symmetric, pure-jump L\'evy processes satisfying the Hartman-Wintner condition. The results should hold in much more general setting, but the scope of the extension is unclear at this moment.
As mentioned, we give applications to the $L^p$-boundedness of Fourier multipliers. 
Namely, we recover the results of \cite{MR3263924}, \cite{MR2918086}, \cite{MR2345912}, where  Fourier multipliers were constructed by tampering with jumps of L\'evy processes with symmetric L\'evy measure.  
Our present approach to Fourier multipliers is simpler than in those papers because
we do not use  Burkholder's inequalities for martingale transforms.
While the approach does not yield sharp constants in $L^p$ comparisons, it should be of interest in applications to multipliers which do not necessarily arise from martingale transforms.   

We note in passing that the approach to Fourier multiplers via square functions has been used in various settings to prove bounds for operators that arise from martingale transforms, such as Riesz transforms and other singular integrals.  For some  recent application of this idea, see \cite[Lemma 1]{MR3306689} and 
\cite[proof of Theorem 1.1]{2015arXiv1506.01208}, where different Littlewood-Paley square functions are employed to prove $L^p$--boundedness for operators arising from martingale transforms.   
We also note that the constants in our $L^p$ estimates of the square functions and Fourier multipliers depend only on $p\in (1,\infty)$ and in particular they do not depend on the dimension of $\Rd$.
It is interesting to note that our applications, unlike those presented in Stein \cite{MR0290095} for his proof of the H\"ormander multiplier theorem, do not depend on pointwise comparisons of Littlewood-Paley square functions before and after applying the multiplier.  Instead, it suffices to have an integral control of the quantities involved, because we can use
the isometry property of the square function on $L^2$ and the usual
pairing to define and study the multiplier. 
In particular, {in applications} we only use two square functions $\tilde G(f)$ and $G(f)$, rather than a whole family of square functions.

The structure of the paper is as follows.  In \S2 we introduce the considered semigroups
and we recall their basic properties.
In \S3 we prove the Hardy-Stein identities. In \S4 we define the square functions and give their upper and lower bounds in $L^p$. In \S5 we present applications to Fourier multipliers.

\section{Preliminaries}

We use ``$:=$" 
to emphasize definitions, e.g., 
$a \wedge b := \min \{ a, b\}$ and $a \vee b := \max \{ a, b\}$. 
For two nonnegative functions $f$ and $g$ on the same domain we write
$f\approx g$ if
there is a positive number $c\geq 1$
such that $c^{-1}\, g \leq f \leq c\, g$ (uniformly for all arguments involved).
All the sets  and functions considered in this work are assumed  real-valued and Borel measurable, unless stated otherwise. 

We 
consider 
the Euclidean space $\Rd$ with dimension $d\geq1$ and the $d$-dimensional Lebesgue measure $dx$. The Euclidean scalar product and norm on $\Rd$ are denoted by $x\cdot y$ and $|x|$.
For every $p\in [1,\infty)$ we let $L^p:=L^p({\Rd},dx)$ be the collection of all the (real-valued Borel-measurable) functions $f$ on $\Rd$ with finite norm
$$
\|f\|_p:=\left[\int_{\Rd} |f(x)|^p dx\right]^{1/p}.
$$
{As usual, $\|f\|_\infty$ denotes the essential supremum of $|f|$.}
For $p=2$ we use the usual scalar product on $L^2$,
$$
\langle f, g\rangle:=\int_{\Rd} f(x)g(x)dx.
$$ 
Let $\nu$ be a measure on $\R^d$ such that $\nu(\left\{0\right\})=0$ and
\begin{equation}\label{e:Lm}\tag{LM}
\int_{\Rd}(1\wedge|y|^2)\nu(dy)<\infty.
\end{equation}
In short: $\nu$ is a L\'evy measure.
We assume that $\nu$ is symmetric: for all (Borel) sets $B\subset \Rd$,
\begin{equation}\label{e:sLm}\tag{S}
\nu(B)=\nu(-B).
\end{equation}
For later convenience we note that given of nonnegativity or absolute integrability of function $k$,
\begin{equation}\label{eq:xxy}
\int \int k(x,y)\nu(dy)dx=
\int \int k(x,-y)\nu(dy)dx=
\int \int k(x+y,-y)\nu(dy)dx.
\end{equation}
Here we used the symmetry of $\nu$, Fubini's theorem and the translation invariance of the Lebesgue measure. In effect the variables in \eqref{eq:xxy} are  changed according to $(x,y,x+y)\mapsto (x+y,-y,x)$.  As a consequence,
\begin{eqnarray}
&&\int \int \ind_{|k(x)|>|k(x+y)|}|k(x+y)-k(x)|\ |h(x+y)-h(x)|\nu(dy)dx \nonumber\\
&=&\frac12 \int \int |k(x)-k(x+y)|\ |h(x)-h(x+y)|\nu(dy)dx,\label{eq:halfp}
\end{eqnarray}
where $k,h$ are arbitrary.
We define
\begin{equation}\label{Lexponent} 
\psi(\xi)=\int_\Rd \left(1-\cos(\xi\cdot x)\right)\nu(dx),\quad \xi\in \Rd, 
\end{equation}
Clearly, $\psi(-\xi)=\psi(\xi)$ for all $\xi$. Finally, we shall assume 
the following Hartman-Wintner condition on $\nu$:
\begin{equation}\label{e:HaWi}\tag{HW}
\lim_{|\xi|\to \infty}
\frac{\psi(\xi)}{\log|\xi|}=\infty. 
\end{equation}
Below we work precisely under these three assumptions \eqref{e:Lm}, \eqref{e:sLm} \eqref{e:HaWi}, except in specialized examples.\\
We let
\begin{equation}\label{e:Fi}
p_t(x)=(2\pi)^{-d}\int_\Rd e^{-i\xi\cdot x}e^{-t\psi(\xi)}d\xi,\quad t>0, \ x\in \Rd.
\end{equation}
Clearly, $p_t(-x)=p_t(x)$ for all $x$ and $t$, and $p_t(x)\le p_t(0)\to 0$ as $t\to \infty$.
By the characterization of the infinitely divisible distributions, i.e. the L\'evy-Khintchine formula, $p_t$ is a density function of a probability measure on $\Rd$ (see \cite{MR2569321} for a direct construction),
$$
\int_\Rd p_t(x)dx=1.
$$
The Fourier transform of $p_t$ is
\begin{equation}\label{e:dFt}
\hat{p_t}(\xi):=\int_\Rd e^{i\xi\cdot x}p_t(x)dx=e^{-t\psi(\xi)}, \quad \xi\in \Rd,\ t>0.
\end{equation}
By \eqref{e:Fi} and \eqref{e:HaWi}, $p_t(x)$ is smooth in $x$ and $t$. 
By \eqref{e:dFt}, $p_t$ form a convolution semigroup of functions:
$$
p_t*p_s=p_{t+s}.
$$
For notational convenience we let
$$
p_t(x,y)=p_t(y-x), \qquad x,y\in \Rd,\ t>0.
$$
From the above discussion we have  the following symmetry property
\begin{equation}\label{eq:psymmetric}
p_t(x,y)=p_t(y,x)\,,\quad x,y\in \Rd,\ t>0,
\end{equation}
the Chapman--Kolmogorov equations
\begin{equation}
  \label{eq:ck}
\int_{\Rd} p_s(x,y)p_t(y,z)dy=p_{s+t}(x,z),\qquad x,z\in \Rd,\; s,t> 0   
\end{equation}
and the  Markovian property 
\begin{equation}\label{eq:1}
\int_{\Rd} p_t(x,y)dy=\int_{\Rd} p_t(x,y)dx=1.
\end{equation}
In fact, $p_t$ is a transition probability density of a symmetric, pure jump L\'evy process $\left\{X_t, t\geq0\right\}$ with values in $\RR^d$ and the characteristic function given by
$$
\E\left[ e^{i\xi\cdot X_t}\right]= e^{-t\psi(\xi)}, \quad 
 t\geq0.
$$
The function $\psi$ is called the {\it characteristic} or {\it L\'evy-Khintchine exponent} of $X_t$.
For an initial state $x\in\Rd$, a Borel set $A\subset\Rd$ and a function $f$ on $\Rd$ we let
$$
\px(X_t\in A):=\p(X_t+x\in A), \quad \quad \Ex f(X_t):=\E f(X_t+x).
$$
It is well-known that
$$
P_t f(x):=\Ex f(X_t)=\int_{\Rd} p_t(x,y)f(y)dy
$$
defines a Feller semigroup on $C_0(\Rd)$, the space of continuous functions on $\Rd$ vanishing at infinity. That is, $P_t C_0(\Rd)\subset C_0(\Rd)$ for all $t>0$, and $(P_t)$ is strongly continuous: $\|P_t f-f\|_{\infty}\to 0$ as $t\to0$ for all $f\in C_0(\Rd)$. We let $L$ be the corresponding infinitesimal generator of $(P_t)$: 
$$
Lf:= \lim_{t\searrow 0}\frac{P_tf-f}{t}.
$$
Here the limit is taken in the supremum norm. Let $\CD(L)$ be the domain of $L$. Then $C^2_0(\Rd)\subset\CD(L)$, where 
$$
C_0^2(\Rd):=\left\{f\in C^2(\Rd)\cap C_0(\Rd):  \frac{\partial f}{\partial x_i}, \frac{\partial^2 f}{\partial x_i\partial x_j}\in C_0(\Rd)
,\quad \ 1\leq i,j\leq d\right\} .
$$
We similarly define the spaces $C_0^k(\Rd)$, $k=1,2,3,\ldots$, and their intersection $C_0^\infty(\Rd)$.
By \cite[Theorem 31.5]{MR3185174} and the symmetry of $\nu$, the generator $L$ satisfies
\begin{equation}\label{PtGen}
Lf(x)=\lim_{\varepsilon\searrow 0}\int_{|y|>\varepsilon}\left(f(x+y)-f(x)\right)\nu(dy),\quad f\in C_0^2(\Rd),\ x\in\Rd.
\end{equation}
By Jensen's inequality and Fubini-Tonelli, $(P_t)$ is also a semigroup of contractions on $L^p$ for every $1\leq p< \infty$, that is, $\|P_tf\|_p\leq \|f\|_p$.
Furthermore, $(P_t)$ is strongly continuous on $L^p$ for every $1\leq p<\infty$. By \cite[Theorem~2.1]{MR3010850} we have 
$p_t(x,\cdot)\in C^{\infty}_0(\Rd)\cap L^1(\Rd)$ for all $x\in\Rd$ and $t>0$. 
In fact it follows from \cite[the proof of Theorem~2.1]{MR3010850} that for fixed $t>0$ and $x\in\Rd$, $p_t(x)=\varphi*\tilde p(x)$, where $\varphi$ is a function in the Schwarz class $\mathcal{S}(\Rd)$, and $\tilde p$ is a probability {measure}. 
Hence, if $f\in L^p(\Rd)$ for some $1\leq p<\infty$, then $P_tf\in L^p(\Rd)\cap C_0^{\infty}(\Rd)$. 
Since $p_t(x,y)\leq p_t(0)$, we have
that $P_t:L^2(\Rd)\to L^{\infty}(\Rd)$ is bounded for all $t>0$. This  property is called ultracontractivity.  For more on this topic, see Davies \cite{MR990239}.

\begin{example}
\rm
The above  assumptions are satisfied 
for  the semigroup 
of many L\'evy processes and  in particular
for 
the semigroup of the isotropic symmetric stable L\'evy processes, associated with the fractional Laplacian.  Indeed, as is 
well-known, the transition density of these processes for $0<\alpha<2$, can be written as 
\begin{equation}\label{subordination}
p_t^{(\alpha)}(x, y)=p_t^{(\alpha)}(x-y)=\int_0^{\infty}\frac{1}{(4\pi s)^{d/2}}e^{\frac{-|x-y|^2}{4s}} \eta^{\alpha/2}_t(s)\,ds,
\end{equation}
where $\eta^{\alpha/2}_t(s)$ is the density for the $\alpha/2$-stable subordinator \cite{MR2569321}.  From this it follows that for each $t>0$, $p_t^{(\alpha)}(x)$ is a radially decreasing function of $x$, and 
$$p_t^{(\alpha)}(x)\leq p_t^{(\alpha)}(0)=\frac{p_1(0)}{t^{\alpha/d}}<\infty.$$
In particular the corresponding semigroup is ultracontractive.  Its L\'evy measure is 
$$\nu(dy)=\mathcal{A}_{d,-\alpha}|y|^{-d-\alpha}dy, \,\,\,\, y\in \Rd,
$$
where
\begin{equation}\label{eq:stableconstant}
\mathcal{A}_{d,-\alpha}=
2^{\alpha}\Gamma\big((d+\alpha)/2\big)\pi^{-d/2}/|\Gamma(-\alpha/2)|.
\end{equation} 
Our assumptions also hold for many other semigroups obtained by subordination of the Brownian motion \cite{MR2978140} and for the more general unimodal L\'evy processes \cite{MR3339224}, provided they satisfy the so-called weak lower scaling condition \cite{MR3339224}. 
\hfill \qed
\end{example}
We shall need 
the following fundamental inequality of Stein \cite{MR0131517} which holds for symmetric Markovian semigroups.  

\begin{lemma}\label{stein}  For $f\in L^p$, $1<p\leq \infty$,
define the maximal function $f^*(x)=\sup_{t}|P_tf(x)|$.
Then,
\begin{equation}\label{maxineq}
\|f^{*}\|_p\leq \frac{p}{p-1}\|f\|_p,
\end{equation}
where the right hand side is just $\|f\|_{\infty}$, 
if $p=\infty$.  
\end{lemma}
We note that Stein \cite{MR0131517} gives 
an unspecified constant 
depending only on $p$ for this inequality. For our applications here this is sufficient, however it is well-known that the inequality actually holds with the explicit constant given above.  In fact, this is nothing more than the constant in Doob's inequality for martingales. The latter is the tool  used in \cite[Chapter 4]{MR0252961} for the proof of the inequality.  For a shorter argument using continuous time martingales and Doob's inequality, we refer the reader to Kim \cite[Proposition 2.3]{2015arXiv1506.01208}.   Kim's proof is the 
zero-potential case of the proof given in Shigekawa \cite{MR1951521} for Feynman--Kac semigroups.  This proof (the zero-potential case of Shigekawa) has been  
known to experts for many years.

\section{Hardy-Stein identity}

The following elementary results are given in \cite{MR3251822}.
Let $1<p<\infty$.
For $a, b\in \RR$ we set
\begin{equation}\label{eq:defF}
 F(a,b) = |b|^p-|a|^p - pa|a|^{p-2}(b-a)\,.
\end{equation}
Here $F(a,b)=|b|^p$ if $a=0$, and $F(a,b)=(p-1)|a|^p$ if $b=0$.
For instance, if $p=2$, then $F(a,b)=(b-a)^2$.
Generally, $F(a,b)$ is the second-order Taylor remainder of $\RR\ni x\mapsto |x|^p$, therefore by convexity, $F(a,b)\geq 0$.
Furthermore, for $1<p<\infty$ and 
$\varepsilon\in \RR$ 
we define
\begin{equation}\label{eq:Feps}
F_\varepsilon(a,b) = 
(b^2+\varepsilon^2)^{p/2}-(a^2+\varepsilon^2)^{p/2}
  - pa(a^2+\varepsilon^2)^{(p-2)/2}(b-a)\,.
\end{equation}
Since $F_\varepsilon (a,b)$ is the second-order Taylor remainder of $\RR\ni x\mapsto (x^2+\varepsilon^2)^{p/2}$, by convexity, $F_\varepsilon(a,b)\geq 0$.
Of course, $F_\varepsilon(a,b)\to F_0(a,b)=F(a,b)$ as $\varepsilon\to 0$.
\begin{lemma}[\cite{MR3251822}]\label{lem:F}
For every $p>1$, we have constants $0<c_p\le C_p<\infty$ such that
\begin{equation}\label{elem-ineq}
c_p (b-a)^2(|b|\vee |a|)^{p-2}\le F(a,b) \le
C_p (b-a)^2(|b|\vee |a|)^{p-2},\qquad
a, b\in\R.
\end{equation} 
If $p\in (1,2)$, then 
\begin{equation}
  \label{eq:ub}
0\le F_\varepsilon(a,b)\leq \frac{1}{p-1}F(a,b)\,,\qquad \varepsilon, a, b\in\R\,.
\end{equation}
\end{lemma}

The main result of this section is the following Hardy-Stein identity.
\begin{theorem}
\label{th:HS}
If $1<p<\infty$ and $f\in L^p(\Rd)$, then
\begin{equation}
\label{eq:HS}
\int_\Rd |f(x)|^pdx=\int_0^\infty \int_\Rd \int_\Rd F(P_t f(x),P_t f(x+y))\nu(dy)dxdt.
\end{equation}
\end{theorem}
\begin{proof} 
We first prove the theorem assuming that $f\in L^p(\Rd)\cap C_0^2(\Rd)$. 
If $2\le p<\infty$, then we proceed as follows. Let $0\le t\le T$ and 
$$\xi(t)=|P_{T-t} f|^p.$$ Then 
$\xi(t)\in C_0^2(\Rd)\subset\CD(L)$ for every $t\in[0,T]$ since $P_tf
\in C_0^2(\Rd)$ for all $t\geq0$. Furthermore, 
if $u\in C_0^2(\Rd)$, then we have
\begin{align*}
\frac{\partial}{\partial x_i} |u|^p&=p|u|^{p-2}u u_i\,,\\
\frac{\partial^2}{\partial x_j\partial x_i}|u|^p&=p(p-1)|u|^{p-2}u_ju_i+p|u|^{p-2}uu_{ji}\,,
\end{align*}
hence $|u|^p\in C_0^2(\Rd)$. Also, $[0,T]\ni t\mapsto \xi(\cdot)(x)$ is of class $C^1$ for every $x\in\Rd$ as it can be seen from the following direct differentiation where $L$ denotes the generator of the semigroup. 
\begin{align}
\frac{d}{dt} \xi(t)(x)&=p P_{T-t}f(x)\left| P_{T-t}f(x)\right|^{p-2} \frac{d}{dt} P_{T-t}f(x)\nonumber\\
&=-p P_{T-t}f(x)\left| P_{T-t}f(x)\right|^{p-2} LP_{T-t}f(x).\label{eq:Xideriv}
\end{align}
We have 
\begin{align}\label{eq:FTofCalc}
P_T |f|^p(x)- |P_T f(x)|^p&=\int_0^T \frac{d}{dt}\left(P_t\xi(t)(x)\right)dt\\
&=\int_0^T \left[P_t \xi'(t)(x)+P_t L\xi(t)(x)\right] dt\label{eq:FTofCalc2}\\
&=\int_0^T P_t\left[ \xi'(t)+ L\xi(t)\right](x) dt.
\end{align}
The equality (\ref{eq:FTofCalc2}) requires some explanation. Following \cite{2014arXiv1411.7907B}, we  have
\begin{align}\label{eq:PtXideriv}
&\frac{P_{t+h}\xi(t+h)(x)-P_t\xi(t)(x)}{h}=P_{t+h}\xi'(t)(x)\\
&+P_{t+h}\left(\frac{\xi(t+h)-\xi(t)}{h}-\xi'(t)\right)(x)+\frac{P_{t+h}\xi(t)(x)-P_t\xi(t)(x)}{h}.
\end{align}
Recall that $\xi(t)\in\CD(L)$. By (\ref{eq:Xideriv}), $\xi'(t)\in C_0(\Rd)$ for every $t\in[0,T]$. 
Furthermore, since $P_t$ is strongly continuous and $p\geq2$, both $LP_{T-t}f$ and $P_{T-t}f\left| P_{T-t}f\right|^{p-2}$ 
are continuous mapping $[0,T]$ to $C_0(\Rd)$. In view of (\ref{eq:Xideriv}), $[0,T]\ni t \mapsto \xi'(t)\in C_0(\Rd)$ is also continuous. Letting $h\to0$ in (\ref{eq:PtXideriv}),  we get (\ref{eq:FTofCalc2}).
We then have
\begin{align}\label{eq:HeatEq}
[\xi'(t)+ L\xi(t)](x)&=\int_\Rd \left\{|P_{T-t} f(x+y)|^p-|P_{T-t} f(x)|^p\right. \\
&\left.-p P_{T-t}f(x)\left| P_{T-t}f(x)\right|^{p-2} \left[P_{T-t}f(x+y)-P_{T-t}f(x)\right]\right\}\nu(dy)\nonumber\\
&=\int_\Rd F(P_{T-t} f(x),P_{T-t} f(x+y))\nu(dy)\nonumber.
\end{align}
Integrating (\ref{eq:FTofCalc}) with respect to $x$ and using \eqref{eq:1} we obtain
\begin{equation*}
\int_\Rd |f(x)|^pdx-\int_\Rd |P_Tf(x)|^pdx=\int_0^T \int_\Rd \int_\Rd F(P_t f(x),P_t f(x+y))\nu(dy)dxdt.
\end{equation*}
But $\int_\Rd |P_Tf(x)|^pdx\to 0$ as $T\to \infty$ because of dominated convergence theorem. Indeed, by \eqref{maxineq} and $|P_Tf(x)|^p\le f^*(x)^p$ for every $x\in \Rd$,
and  for  $q=p/(p-1)$ by H\"older inequality we have
\begin{equation}
\left|\int p_T(x,y)f(y)dy\right|\le \|f\|_p\left(\int_\Rd p_T(x,y)^q dy\right)^{1/q},
\end{equation}
whereas $\int_\Rd p_T(x,y)^q dy\le \sup_{x,y\in \Rd} p_T(x,y)^{q-1}\to 0$ as $T\to \infty$. Thus, \eqref{eq:HS} follows. 

Suppose $1<p<2$. For $0\le t\le T$ and $\varepsilon>0$ we define
$$
\xi_{\varepsilon}(t)=\left((P_{T-t}f)^2+\varepsilon^2\right)^{p/2}-\varepsilon^p.
$$
As in the case $2\leq p<\infty$, we conclude that $\xi_{\varepsilon}(t)\in C_0^2(\Rd)\subset\CD(L)$ for every $t\in[0,T]$. 
Indeed, for any $u\in C_0^2(\Rd)$ we have
\begin{align*}
\frac{\partial}{\partial x_i} \left(u^2+\varepsilon^2\right)^{p/2}&=p\left(u^2+\varepsilon^2\right)^{(p-2)/2}u u_i\,,\\
\frac{\partial^2}{\partial x_j\partial x_i}\left(u^2+\varepsilon^2\right)^{p/2}&=p(p-2)\left(u^2+\varepsilon^2\right)^{(p-4)/2}u^2u_ju_i\,\\
&+p\left(u^2+\varepsilon^2\right)^{(p-2)/2}(u_ju_i+uu_{ji})\,.
\end{align*}
Furthermore, $[0,T]\ni t\mapsto \xi_{\varepsilon}(\cdot)(x)$ is also of class $C^1$ for every $x\in\Rd$, and
\begin{align*}
\frac{d}{dt} \xi_{\varepsilon}(t)(x)&=p P_{T-t}f(x)\left[(P_{T-t}f(x))^2+\varepsilon^2\right]^{(p-2)/2} \frac{d}{dt} P_{T-t}f(x)\\
&=-p P_{T-t}f(x)\left[(P_{T-t}f(x))^2+\varepsilon^2\right]^{(p-2)/2} LP_{T-t}f(x).
\end{align*}
Therefore $\xi'_{\varepsilon}(t)\in C_0(\Rd)$ for every $t\in[0,T]$, and $[0,T]\ni t\mapsto \xi'_{\varepsilon}(t)\in C_0(\Rd)$ is continuous. We have 
\begin{align}\label{eq:FTofCalcEps}
P_T \left((f^2+\varepsilon^2)^{p/2}\right)(x)- \left((P_Tf(x))^2+\varepsilon^2\right)^{p/2}
=\int_0^T \frac{d}{dt}\left(P_t\xi_{\varepsilon}(t)(x)\right)dt\\
=\int_0^T \left[P_t \xi'_{\varepsilon}(t)(x)+P_t L\xi_{\varepsilon}(t)(x)\right] dt=\int_0^T P_t\left[ \xi'_{\varepsilon}(t)+ 
L\xi_{\varepsilon}(t)\right](x) dt.
\end{align}
Consequently,
\begin{align*}
[\xi'_{\varepsilon}(t)+ L\xi_{\varepsilon}(t)](x)
&=\int_\Rd \left\{\left((P_{T-t}f(x+y))^2+\varepsilon^2\right)^{p/2}-\left((P_{T-t}f(x))^2+\varepsilon^2\right)^{p/2}\right. \\
&\left.-p P_{T-t}f(x)\left((P_{T-t}f(x))^2+\varepsilon^2\right)^{(p-2)/2} \left[P_{T-t}f(x+y)-P_{T-t}f(x)\right]\right\}\nu(dy)\\
&=\int_\Rd F_{\varepsilon}(P_{T-t} f(x),P_{T-t} f(x+y))\nu(dy).
\end{align*}
Integrating (\ref{eq:FTofCalcEps}) with respect to $x$ we obtain
\begin{align*}
&\int_\Rd\left[P_T \left((f^2+\varepsilon^2)^{p/2}\right)(x)-\left((P_Tf(x))^2+\varepsilon^2\right)^{p/2}\right]dx\\
&=\int_\Rd \left((f(x)^2+\varepsilon^2)^{p/2}-\varepsilon^p\right)dx
-\int_\Rd\left[\left((P_Tf(x))^2+\varepsilon^2\right)^{p/2}-\varepsilon^p\right]dx\\
&=\int_0^T\int_\Rd\int_\Rd F_{\varepsilon}(P_t f(x),P_t f(x+y))\nu(dy)dxdt.
\end{align*}
Note that the expression above is finite and uniformly bounded with respect to $T$ and $\varepsilon$. Indeed, since $0<p/2<1$, the function $x\mapsto x^{p/2}$ is $p/2$-H\"older continuous on $[0,\infty)$, we have
$$
(f(x)^2+\varepsilon^2)^{p/2}-\varepsilon^p\leq c_p |f(x)|^p,
$$
and
$$
\left((P_Tf(x))^2+\varepsilon^2\right)^{p/2}-\varepsilon^p\leq c_p|P_Tf(x)|^p.
$$
Let $\varepsilon\to0$. In view of (\ref{eq:ub}) and dominated convergence (see also \cite[Remark 7]{MR3251822}) we get
\begin{equation*}
\int_\Rd |f(x)|^pdx-\int_\Rd |P_Tf(x)|^pdx=\int_0^T \int_\Rd \int_\Rd F(P_t f(x),P_t f(x+y))\nu(dy)dxdt.
\end{equation*}
Using the same argument as in the previous part we get $\int_\Rd |P_Tf(x)|^pdx\to0$ as $T\to\infty$. 
This together with the previous case gives \eqref{eq:HS} for all $1<p<\infty$ and 
$f\in L^p(\Rd)\cap C_0^2(\Rd)$.

We next relax the assumption that $f\in  L^p(\Rd)\cap C_0^2(\Rd)$.  For $1<p<\infty$ and general $f\in L^p(\Rd)$ we let $s>0$. Then $P_sf\in L^p(\Rd)\cap C_0^{\infty}(\Rd)$, and so by the preceding, discussion
\begin{equation}\label{eq:HS2}
\int_\Rd |P_s f(x)|^pdx=\int_s^\infty \int_\Rd \int_\Rd F(P_t f(x),P_t f(x+y))\nu(dy)dxdt.
\end{equation}
By the strong continuity of $P_t$ in $L^p(\Rd)$, the left-hand side of \eqref{eq:HS2} tends to $\|f\|^p_p$ as $s\to 0$. The right-hand side also converges as $s\to 0$.
The theorem follows.
\end{proof}

\section{
Square functions} 
For $f\in L^1(\Rd)\cup L^\infty(\Rd)$ we let
$$
G(f)(x):=\left(\int_0^{\infty}\int_{\RR^d}(P_tf(x+y)-P_tf(x))^2\nu(dy)dt\right)^{1/2},
$$
and
$$
\widetilde{G}(f)(x):=\left(\int_0^{\infty}\int_{\left\{|P_tf(x)|>|P_tf(x+y)| \right\}}(P_tf(x+y)-P_tf(x))^2\nu(dy)dt\right)^{1/2}.
$$
Clearly, $0\leq \widetilde{G}(f)(x)\leq G(f)(x)$ for every $x$. By (\ref{eq:HS}) and the symmetry, 
\begin{equation}\label{eq:L2isometry}
\|f\|_2^2=\|G(f)\|_2^2=2\|\widetilde{G}(f)\|_2^2.
\end{equation}
By polarization, for $f,g\in L^2(\Rd)$ we have
\begin{equation}\label{eq:pol}
\langle f, g\rangle=
\int_\Rd \int_0^{\infty}\int_{\RR^d}[P_tf(x+y)-P_tf(x)]\,[P_tg(x+y)-P_tg(x)]\ \nu(dy)dtdx. 
\end{equation}

The main result of this section is the following theorem.
\begin{theorem}\label{thtildeGp}
Let $1< p<\infty$. There is a constant $C$ depending only on $p$ 
such that
\begin{equation}\label{eq:tildeGp}
C^{-1}\|f\|_p\leq\|\widetilde{G}(f)\|_p\leq C\|f\|_p, \qquad f\in L^p(\RR^d).
\end{equation}
\end{theorem}
The result is proved below after a sequence of partial results. In another direction, at the end of this section we show in Example~\ref{eq:incGG} that $G$ is too large to give a characterization of $L^p(\Rd)$ for $1<p<2$. Nevertheless, 
$\|G(f)\|_p\leq C_p\|f\|_p$, for $2\leq p<\infty$, as we now prove
by using the Burkholder-Gundy inequalities.

We start by introducing the Littlewood-Paley function  $G_{*}$  which is
the conditional expectation of 
the quadratic variation of a martingale.  For classical harmonic functions in the upper half-space of $\R^d$, such objects have appeared many times in the literature, see for example  \cite{MR855179}, \cite{MR797052}. 
The construction for the generalized Poisson semigroups is presented in \cite{MR583240}.  
Here we simply fix $f\in C_c^{\infty}(\RR^d)$, $T>0$, and let 
$$M_t=P_{T-t}f(X_t)-P_Tf(z), \quad 0<t<T.$$ 
When the process $X_t$ starts at $z\in \Rd$, $M_t$ is a martingale starting at $0$.
Such  space-time ({\it parabolic}) 
martingales were first 
used for the Brownian motion in Ba\~nuelos and M\'endez-Hern\'andez  \cite{MR2001941} to study martingale transforms 
that lead  to 
Fourier multipliers related to the Beurling-Ahlfors operator.  They were then applied
in  \cite{MR2345912, MR3263924} 
to more general L\'evy processes.  
We recall the properties of $M_t$ here to clarify the use of the Burkholder-Gundy inequality and to elucidate the origins of our Littlewood-Paley square functions.   For full details, we refer the reader to \cite{MR3263924}.

Applying the It\^o formula (see \cite[p.~1118]{MR3263924}, where this is done for general L\'evy processes) we have that 
\begin{eqnarray}\label{levyito}
M_t = \int_{0}^{t}\int_{\R^d}\left[P_{T-s}f(X_{s-}+y)
- P_{T-s}f(X_{s-})\right]\tilde{N}(ds,dy), \,\,\,\, 0<t<T. 
\end{eqnarray}
Here  $$ {\tilde N}(t,A) = N(t,A) - t \nu(A),$$ 
and  $N$ is a Poisson random measure on ${\R}^{+} \times \R^d$ with intensity
measure $dt\times d\nu$.  In fact we take
$$N(t, A) = \#\{0 \leq s \leq t,\Delta X_s \in A\}, \qquad t \geq 0,\ A \subset \R^d,$$
where $\Delta X_s=X_s-X_{s^{-}}$ denotes the jump of the process at time $s > 0$. The quadratic variation of $M_t$  is
\begin{eqnarray*}
[M]_t= \int_{0}^{t}\int_{\R^d}|P_{T-s}f(X_{s-}+y)
- P_{T-s}f(X_{s-})|^2d\nu{(y)\,ds}.
\end{eqnarray*}
For a slightly different representation of \eqref{levyito} without using the process $N$, and 
for references to It\^o's formula for processes with jumps, see \cite[p.~847]{MR2928339}.

We now define
\begin{eqnarray*}
G_{*}(f)(x)=\left(\int_{0}^{\infty}\int_{\R^d}\int_{\R^d}|P_tf(z+y)-P_tf(z)|^2 p_t(x, z)dz\nu(dy) dt\right)^{1/2},
\end{eqnarray*}
and
\begin{eqnarray*}
G_{*,T}(f)(x)=\left(\int_{0}^{T}\int_{\R^d}\int_{\R^d}|P_tf(z+y)-P_tf(z)|^2 p_t(x, z)dz\nu(dy) dt\right)^{1/2}.
\end{eqnarray*}
Notice that
$G_{*,T}(f)(x)\nearrow G_{*}(f)(x)$ as $T\to\infty$. We claim that 
\begin{equation}\label{eq:G*conditional}
G^2_{*, T}(f)(x)=
\int_{\R^d}\E_{z}^{x}\left(\int_{0}^{T}\int_{\R^d}|P_{T-s}f(X_s+y)-P_{T-s}f(X_s)|^2\nu(dy)ds\right)p_{T}(z,x) dz,
\end{equation}
where 
\begin{align*}
\E_{z}^{x}&\left(\int_{0}^{T}\int_{\R^d}|P_{T-s}f(X_s+y)-P_{T-s}f(X_s)|^2\nu(dy)ds\right)\\
:=\E_{z}&\left(\int_{0}^{T}\int_{\R^d}|P_{T-s}f(X_s+y)-P_{T-s}f(X_s)|^2\nu(dy)ds\,\big|\,X_T=x\right),
\end{align*}
cf. below. Thus, 
$$
G^2_{*, T}(f)(x)=\int_{\R^d}\E_{z}\left([M]_T\,\big|\,X_T=x\right)p_{T}(z,x) dz=\int_{\R^d}\left(\E^x_{z}[M]_T\right) \, p_{T}(z,x) dz.
$$
The proof of  \eqref{eq:G*conditional}  is exactly the same as the proof for harmonic functions in the upper half-space of $\R^d$ given in \cite[p.~663]{MR855179}. (See \cite{MR583240} for the more general construction for Poisson semigroups.) 
Indeed, by the definition of the conditional distribution of $X_s$ under $\p_z$ given $X_T=x$, we have
\begin{align*}
&\int_{\R^d}\E_{z}^{x}\left(\int_{0}^{T}\int_{\R^d}|P_{T-s}f(X_s+y)-P_{T-s}f(X_s)|^2\nu(dy)ds\right)p_{T}(z,x) dz\\
=&\int_{\R^d}\left(\int_{0}^{T}\int_{\R^d}\frac{p_s(z,w)p_{T-s}(w,x)}{p_T(z,x)}\int_{\R^d}|P_{T-s}f(w+y)-P_{T-s}f(w)|^2\nu(dy)dwds\right)p_{T}(z,x)dz\\
=&\int_{0}^{T}\int_{\R^d}p_{T-s}(w,x)\int_{\R^d}|P_{T-s}f(w+y)-P_{T-s}f(w)|^2\nu(dy)dwds\\
=&\int_{0}^{T}\int_{\R^d}\int_{\R^d}|P_{s}f(w+y)-P_{s}f(w)|^2p_{s}(x,w)dw\nu(dy)ds=G^2_{*, T}(f)(x).
\end{align*}
With \eqref{eq:G*conditional} established, we now apply the martingale inequalities to prove that $\|G_{*}(f)\|_p\leq C_p\|f\|_p$ for $2\leq p<\infty$, which also yields the same result for $G(f)$.
\begin{lemma}\label{GfEstimate}
Let $2\le p<\infty$. There is a constant $C$ depending only on $p$ 
such that
$\|G(f)\|_p\leq C\|f\|_p$ for every $f\in L^p(\RR^d)$. 
\end{lemma}

\begin{proof}
  Since $p\geq2$, by Jensen's inequality we get
\begin{align*}
&\int_{\R^d}G^p_{*, T}(f)(x)dx\\
\leq &\int_{\R^d}\int_{\R^d}\E_{z}^{x}\left(\int_{0}^{T}\int_{\R^d}|P_{T-s}f(X_s+y)-P_{T-s}f(X_s)|^2\nu(dy)ds\right)^{p/2}p_{T}(z,x) dzdx\\
=&\int_{\R^d}\E_{z}\left(\int_{0}^{T}\int_{\R^d}|P_{T-s}f(X_s+y)-P_{T-s}f(X_s)|^2\nu(dy)ds\right)^{p/2}dz.
\end{align*}
By the Burkholder-Gundy inequality 
the last term above is less than
\begin{align*}
&C_p\int_{\R^d}\E_{z}\left|f(X_T)-P_Tf(z)\right|^pdz\leq C_p\int_{\R^d}(\E_{z}|f(X_T)|^p+P_T|f|^p(z))dz\\
= &C_p\int_{\R^d}P_T|f(z)|^pdz=C_p\|f\|^p_p.
\end{align*}
By the monotone convergence,
$$
\int_{\R^d}G^p_{*}(f)(x)dx=\lim_{T\to\infty}\int_{\R^d}G^p_{*, T}(f)(x)dx \leq C_p\|f\|^p_p.
$$
We claim that $G(f)(x)\leq\sqrt{2}G_{*}(f)(x)$. Indeed,
by the semigroup property and Jensen's inequality,
\begin{eqnarray*}
G^2(f)(x)&=&\int_{0}^{\infty}\int_{\R^d}|P_tf(x+y)-P_tf(x)|^2\nu(dy) dt\\
&=&\int_{0}^{\infty}\int_{\R^d}|P_{t/2}P_{t/2}f(x+y)-P_{t/2}P_{t/2}f(x)|^2\nu(dy)dt\nonumber\\
&\leq &\int_{0}^{\infty}\int_{\R^d}P_{t/2}|P_{t/2}f(x+y)-P_{t/2}f(x)|^2\nu(dy)dt\nonumber\\
&=&\int_{0}^{\infty}\int_{\R^d}\int_{\R^d}|P_{t/2}f(z+y)-P_{t/2}f(z)|^2 p_{t/2}(x, z)dz\nu(dy)dt\nonumber\\
&=&2\int_{0}^{\infty}\int_{\R^d}\int_{\R^d}|P_{t}f(z+y)-P_{t}f(z)|^2 p_{t}(x, z)dz\nu(dy)dt\nonumber.
\end{eqnarray*}
This completes the proof of the lemma for $f\in C_c^{\infty}(\RR^d)$.  
For arbitrary $f\in L^p(\Rd)$, we choose $f_n\in C_c^{\infty}(\Rd)$ such that $f_n\to f$ in $L^p$. The inequality $\|G(f)\|_p\leq C\|f\|_p$ follows from Fatou's lemma.
\end{proof}
For every $2\leq p<\infty$ and 
$f\in L^p(\RR^d)$ we have by (\ref{eq:HS}), (\ref{elem-ineq}) and \eqref{eq:halfp},
\begin{eqnarray}
\|f\|_p^p&\asymp& \int_{\RR^d}\int_0^{\infty}\int_{\RR^d}(P_tf(x+y)-P_tf(x))^2
(|P_tf(x+y)|\vee |P_tf(x)|)^{p-2}\nu(dy)dtdx\nonumber \\
&=&2 \int_{\RR^d}\int_0^{\infty}\int_{\left\{|P_tf(x)|>|P_tf(x+y)| \right\}}(P_tf(x+y)-P_tf(x))^2
|P_tf(x)|^{p-2}\nu(dy)dtdx \nonumber\\
&
\leq& 2\int_{\RR^d}f^*(x)^{p-2}\widetilde{G}(f)(x)^2dx.\label{eq:fsGt}
\end{eqnarray}
\begin{lemma}\label{tildeGp>2}
Suppose $2\le p<\infty$. There is a constant $C$ depending only on $p$ 
such that
\begin{equation}\label{ineq:LP1}
C^{-1}\|f\|_p\leq\|\widetilde{G}(f)\|_p\leq C\|f\|_p, \qquad f\in L^p(\RR^d).
\end{equation}
\end{lemma}
\begin{proof}
Since $\widetilde{G}(f)(x)\leq G(f)(x)$, the right-hand side of (\ref{ineq:LP1}) follows immediately from Lemma~\ref{GfEstimate}.
By H\"older's inequality,
$$
\int_{\RR^d}f^*(x)^{p-2}\widetilde{G}(f)(x)^2dx\leq\left[\int_{\RR^d}(f^*(x)^{p-2})^{\frac{p}{p-2}}dx\right]^{\frac{p-2}{p}}
\left[\int_{\RR^d}(\widetilde{G}(f)(x)^{2})^{\frac{p}{2}}dx\right]^{\frac{2}{p}}
$$
$$
=\|f^*\|_p^{p-2}\|\widetilde{G}(f)\|_p^{2}\leq C\|f\|_p^{p-2}\|\widetilde{G}(f)\|_p^{2}.
$$
By \eqref{eq:fsGt} and \eqref{maxineq} we get $\|f\|_p^p\leq C\|f\|_p^{p-2}\|\widetilde{G}(f)\|_p^{2}$, which yields the result.
\end{proof}

Combining Lemma~\ref{GfEstimate} and Lemma \ref{tildeGp>2} we obtain the following.
\begin{corollary}\label{Gp>2}
Suppose $2\le p<\infty$. There is a constant $C$ depending only on $p$ 
such that
\begin{equation}\label{ineq:Gp>2}
C^{-1}\|f\|_p\leq\|G(f)\|_p\leq C\|f\|_p, \qquad f\in L^p(\RR^d).
\end{equation}
\end{corollary}

We now discuss the regime $1<p<2$. 

\begin{lemma}\label{lem:seGtpsmall}
Suppose $1<p<2$. There is a constant $C$ depending only on $p$ such that 
\begin{equation}\label{eq:1p2feGf}
C^{-1}\|f\|_p\leq\|\widetilde{G}(f)\|_p\leq C\|f\|_p, \qquad f\in L^p(\RR^d).
\end{equation}
\end{lemma}
\begin{proof}
We first consider the right-hand inequality. Our proof proceeds exactly as the proof in \cite[pp. 87-88]{MR0290095} for the boundedness of the Littlewood-Paley square function $g$ in the range $1<p<2$.  Here, however, instead of using the Hardy-Littlewood maximal function and its boundedness on $L^p(\Rd)$,  we  use the maximal function of the semigroup and Lemma \ref{stein}.  Also, in place of the identity Lemma 2 of \cite[p.88]{MR0290095}, we use our Hardy-Stein identity. More precisely, setting 
$$
I(x)=
\int_0^{\infty}\int_{\left\{|P_tf(x)|>|P_tf(x+y)| \right\}}(P_tf(x+y)-P_tf(x))^2|P_tf(x)|^{p-2}\nu(dy)dt
$$
we have, by \eqref{eq:HS} of Theorem \ref{th:HS} and \eqref{elem-ineq} of Lemma \ref{lem:F}, that there exists a constant $C_p$ depending only on $p$  such that 
\begin{equation}\label{HS-bound}
\int_{\Rd}I(x)dx\leq C_p\int_{\Rd}|f(x)|^p dx. 
\end{equation}
Now observe that 
\begin{eqnarray*}
\widetilde{G}(f)(x)^2&=&\int_0^{\infty}\int_{\left\{|P_tf(x)|>|P_tf(x+y)| \right\}}(P_tf(x+y)-P_tf(x))^2\nu(dy)dt\\
&=&\int_0^{\infty}\int_{\left\{|P_tf(x)|>|P_tf(x+y)| \right\}}(P_tf(x+y)-P_tf(x))^2|P_tf(x)|^{p-2}|P_tf(x)|^{2-p}\nu(dy)dt\\
&\leq & f^{*}(x)^{2-p}I(x), 
\end{eqnarray*}
where we used the fact that $1<p<2$.    With $r=2/(2-p)$ and $r'=2/p$ so that $1<r, r'<\infty$ and $1/r+1/r'=1$, we can integrate both sides of this inequality and apply H\"older's inequality to obtain 
\begin{eqnarray*}
\int_{\Rd} \widetilde{G}(f)(x)^p dx &\leq & \int_{\R^d} {f^{*}(x)}^{\frac{p(2-p)}{2}}I(x)^{p/2}\,dx\\
&\leq &  \left(\int_{\R^d} {f^{*}(x)}^{p}dx\right)^{(2-p)/2}\left(\int_{\Rd}I(x)dx\right)^{p/2}\\
&\leq& \left(\frac{p}{p-1}\right)^{\frac{p(2-p)}{2}} C_p^{p/2}\left(\int_{\Rd}|f(x)|^p\,dx\right)^{(2-p)/2}\left(\int_{\Rd}|f(x)|^p\,dx\right)^{p/2}\\
&=&\left(\frac{p}{p-1}\right)^{\frac{p(2-p)}{2}} C_p^{p/2}\int_{\Rd}|f(x)|^p\,dx, 
\end{eqnarray*}
where in the last inequality we used Lemma \ref{stein} and the Hardy-Stein bound \eqref{HS-bound}.    This gives 
\begin{equation}
\|\widetilde{G}(f)\|_p\leq \left(\frac{p}{p-1}\right)^{\frac{(2-p)}{2}} C_p^{1/2}\|f\|_p, \qquad 1<p<2. 
\end{equation} 

In order to prove the left-hand side of \eqref{eq:1p2feGf}, we fix nonzero $f\in L^p(\Rd)$ and let $s>0$. Define $f_s:= P_sf$ and $g_s:=|f_s|^{(p-1)}\sgn f_s$. By ultracontractivity,
$f_s\in L^2(\Rd)$ and $g_s\in L^{\infty}(\Rd)$. 
Furthermore, for $q=p/(p-1)$ we have
$\|g_s\|_q=\|f_s\|^{p-1}_p$ and 
$$
\|f_s\|^{p}_p=\int_{\RR^d}f_s(x)g_s(x)dx.
$$
Let $\varphi_n:=\ind_{B(0,n)}g_s$. Since $g_s$ is bounded, $\varphi_n\in L^2(\Rd)$ for all $n\geq1$. By (\ref{eq:L2isometry}) and \eqref{eq:halfp}, 
\begin{align*}
&\int_{\RR^d}f_s(x)\varphi_n(x)dx=\frac14(\|f_s+\varphi_n\|^2_2-\|f_s-\varphi_n\|^2_2)=
\frac14(\|G(f_s+\varphi_n)\|^2_2-\|G(f_s-\varphi_n)\|^2_2)\\
=&\int_{\RR^d}\int_0^{\infty}\int_{\RR^d}(P_tf_s(x+y)-P_tf_s(x))(P_t\varphi_n(x+y)-P_t\varphi_n(x))\nu(dy)dtdx\\
=&2\int_{\RR^d}\int_0^{\infty}\int_{\left\{|P_tf_s(x)|>|P_tf_s(x+y)|\right\}}(P_tf_s(x+y)-P_tf_s(x))(P_t\varphi_n(x+y)-P_t\varphi_n(x))\nu(dy)dtdx\\
\leq &2\int_{\RR^d}\widetilde{G}(f_s)(x)G(\varphi_n)(x)dx\leq 2\|\widetilde{G}(f_s)\|_p\|G(\varphi_n)\|_q.
\end{align*}
In the last line we used the Cauchy-Schwarz inequality and H\"older's inequality. Finally, since $q>2$, by Lemma \ref{GfEstimate} we have $\|G(\varphi_n)\|_q\leq C\|\varphi_n\|_q$ and so
\begin{equation}\label{eq:Gpsi}
\int_{\RR^d}f_s(x)\varphi_n(x)dx\leq C\|\widetilde{G}(f_s)\|_p\|\varphi_n\|_q.
\end{equation}
By the monotone convergence, $\|\varphi_n\|_q\to\|g_s\|_q$ as $n\to\infty$, and the left-hand side of \eqref{eq:Gpsi} converges to $\|f_s\|^{p}_p$. This gives
$$
\|f_s\|^{p}_p\leq C\|\widetilde{G}(f_s)\|_p\|g_s\|_q=C\|\widetilde{G}(f_s)\|_p\|f_s\|^{p-1}_p.
$$
Dividing by $\|f_s\|^{p-1}_p$ we obtain $\|f_s\|_p\leq C\|\widetilde{G}(f_s)\|_p$. We let $s\to0$ in 
$$
\widetilde{G}(f_s)=\left(\int_s^{\infty}\int_{\left\{|P_tf(x)|>|P_tf(x+y)| \right\}}(P_tf(x+y)-P_tf(x))^2\nu(dy)dt\right)^{1/2}.
$$
The monotone convergence and strong continuity of $P_t$ in $L^p(\Rd)$ yield \eqref{eq:1p2feGf}.
\end{proof}

\begin{proof}[Proof of Theorem~\ref{thtildeGp}]

The result combines Lemma~\ref{tildeGp>2} and Lemma~\ref{lem:seGtpsmall}.
\end{proof}
It is well-known that the classical  Littlewood-Paley operator $G_{*}$ constructed from harmonic  functions is not 
bounded on $L^p$, if $1<p<2$.   An explicit example for this failure is presented in  \cite{MR797052}.  Inspired by \cite{MR797052} we show that the square operator $G$ also fails to be bounded on $L^p$, if $1<p<2$. Thus, $\tilde G$ and $G$ differ significantly.

\begin{example}\label{eq:incGG}
\rm
For $d\geq2$ and $x\in\Rd$ we let $h(x)=|x|^{-(d+1)/2}$ and $f(x)=h(x)\ind_{|x|\leq 1}$. We have that $f\in L^p(\Rd)$ for $1<p<2d/(d+1)$. Let $P_t$ be the rotationally invariant Cauchy (Poisson) semigroup on $\Rd$.  That is,  the semigroup of the $\alpha$-stable processes with $\alpha=1$ with transition density 
$$
p_t(x,y)=\cc_d\frac{t}{\left(t^2+|x-y|^2\right)^{\frac{d+1}{2}}},
$$
where $\cc_d=\Gamma((d+1)/2)\pi^{-(d+1)/2}$. Since $h$ is locally integrable on $\Rd$ and vanishes at infinity, the function 
$$
v(x,t):=\begin{cases}
		P_th(x), & \ x\in\Rd, \ t>0,\\
		h(x), & \ x\in\Rd, \ t=0,
		\end{cases}
$$ 
is well defined and continuous except at $(x,t)=(0,0)$. We see that $v$ is the classical harmonic extension of $h$ to the upper half-space in $\RR^{d+1}$.  
For $x\in\Rd$ and $s,t>0$ we let 
$$
v_s(x,t)=
\int_{B(0,1/s)}
{p_t(x,y)}h(y)dy.
$$
From scaling it follows that 
\begin{align*}
P_tf(x)&=\cc_d\int_{B(0,1)}\frac{t}{\left(t^2+|x-y|^2\right)^{\frac{d+1}{2}}}h(y)dy\\
&=t^{-(d+1)/2}\cc_d\int_{B(0,1)}\frac{1}{t^d\left(1+|x/t-y/t|^2\right)^{\frac{d+1}{2}}}\frac{1}{|y/t|^{\frac{d+1}{2}}}dy\\
&=t^{-(d+1)/2}v_t(x/t,1),
\end{align*}
and that 
$$
v(x,t)=t^{-(d+1)/2}v(x/t,1), \qquad x\in \Rd,\, t>0.
$$
We have
\begin{align}
G(f)(x)^2&=\mathcal{A}_{d,-1}\int_0^{\infty}\int_{\RR^d}\frac{(P_tf(y)-P_tf(x))^2}{|x-y|^{d+1}}dydt\nonumber\\	
&=\mathcal{A}_{d,-1}\int_0^{\infty}\int_{\RR^d}\frac{(v_t(y/t,1)-v_t(x/t,1))^2}{t^{d+1}|x-y|^{d+1}}dydt\nonumber\\	
&=\mathcal{A}_{d,-1}\int_0^{\infty}\int_{\RR^d}\frac{(v_t(z,1)-v_t(x/t,1))^2}{t|x-tz|^{d+1}}dzdt,\label{GCauchy}
\end{align}
where $\mathcal{A}_{d,-1}$ is the constant in \eqref{eq:stableconstant}. Observe that $v_t(z,1)\nearrow v(z,1)>0$,  for all $z\in\Rd$, as $t\searrow 0$. Furthermore, 
$$
v_t(x/t,1)\leq v(x/t,1)=t^{\frac{d+1}{2}}v(x,t)
\to 0  \quad \mbox{for } t\searrow 0,\, 
x\neq0.
$$
Applying Fubini's theorem in (\ref{GCauchy}) we see that $G(f)\equiv\infty$.
On the other hand, $\tilde G(f)\in L^p$ for every $1<p<2d/(d+1)$, as follows from Theorem~\ref{thtildeGp}.
\hfill $\qed$
\end{example}

\section{Application to L\'evy multipliers}

Among the many applications of classical square functions are those to Fourier multipliers. Accordingly, in this section we prove $L^p$ boundedness for a class of Fourier multipliers that arise in connection to L\'evy processes.  The multipliers were first studied  in \cite{MR2345912} and subsequently in \cite{MR2918086} and \cite{MR3263924} where explicit $L^p$  bounds were proved by using
Burkholder's sharp inequalities for martingale transforms.
These multipliers include the differences of second order Riesz transforms, $R_1^2-R_2^2$, for which the
bounds given in \cite{MR2918086} and \cite{MR3263924}
were already known to be best possible.  Below we derive $L^p$ boundedness of the operators in a different 
way by using our square function inequalities and the  representation of Fourier multipliers from \cite{MR3263924}.

As previously, we consider a symmetric pure-jump L\'evy process $\{X_t, t\geq0\}$ on $\Rd$ with the semigroup $(P_t)$ and (symmetric) L\'evy measure $\nu$ satisfying \eqref{e:HaWi}.
Recall from \eqref{Lexponent} that 
the  L\'evy-Khintchine exponent is 
$$
\psi(\xi)=\int_\Rd \left(1-\cos(\xi\cdot x)\right)\nu(dx),\quad \xi\in \Rd.
$$
Let $\phi(t,y)$ be a bounded 
function on $(0,\infty)\times\Rd$. Let
$1<p,q<\infty$ and $\frac{1}{p}+\frac{1}{q}=1$.
For $f\in L^p(\RR^d)\cap L^2(\RR^d)$ and $h\in L^q(\RR^d)\cap L^2(\RR^d)$, 
we consider
\begin{eqnarray}\label{eq:pairing}
\Lambda(f,h)&=&\int_{0}^{\infty}\int_{\RR^d}\int_{\RR^d}[P_{t}f(x+y)
- P_{t}f(x)]
[P_{t}h(x+y)
- P_{t}h(x)]
 \phi(t,y)\nu(dy)dx dt.
\end{eqnarray}
Although not needed for our argument here, it should be pointed out that this quantity arises in \cite{MR3263924} from the It\^o isometry after taking inner products of the martingale transform of $f$ by the function $\phi$ and the martingale corresponding to $h$ (see \cite[Theorem 3.4]{MR3263924} for more details on this pairing).  Here we just observe that the integral is absolutely convergent, by \eqref{eq:L2isometry} and Cauchy-Schwarz inequality. By \eqref{eq:halfp} and Cauchy-Schwarz,
\begin{eqnarray*}
|\Lambda(f,h)|&\le&\|\phi\|_{\infty}
\int_{0}^{\infty}\int_{\RR^d}\int_{\RR^d}|P_{t}f(x+y)
- P_{t}f(x)|\ 
|P_{t}h(x+y)
- P_{t}h(x)|
\nu(dy)\,dx dt\\
 &= &2\|\phi\|_{\infty}\int_{\Rd}\int_0^{\infty} \int_{\left\{|P_tf(x)|>|P_tf(x+y)|\right\}} |P_{t}f(x+y)- P_{t}f(x)||P_{t}h(x+y)- P_{t}h(x)|
\nu(dy)dtdx\\
 &\leq & 2\|\phi\|_{\infty} \int_{\RR^d} \widetilde{G}(f)(x) G(h)(x)\,dx. 
\end{eqnarray*}
Assuming $1<p\le 2$, we have $2\le q<\infty$, and by H\"older inequality and Theorem~\ref{thtildeGp} we get
\begin{eqnarray*}
|\Lambda(f,h)| &\leq &2\|\phi\|_{\infty} \|\widetilde{G}(f)\|_p\|G(h)\|_q
\leq C_p\|\phi\|_{\infty}\|f\|_p\|h\|_q.   
\end{eqnarray*}
If $2<p<\infty$, then  $1< q<2$, and we similarly have 
\begin{eqnarray*}
|\Lambda(f,h)| &\leq &2\|\phi\|_{\infty} \|G(f)\|_p\|\widetilde G(h)\|_q
\leq C_p\|\phi\|_{\infty}\|f\|_p\|h\|_q.   
\end{eqnarray*}
By the Riesz representation theorem, there is a unique linear operator $S_{\phi}$ on $L^p(\RR^d)$ such that $\Lambda(f,g)=(S_{\phi}f,g)$, and $\|S_\phi\|\le C_p \|\phi\|_{\infty}$.

The computation of the symbol of the multiplier is now exactly as in \cite[p.1134]{MR3263924} where it is done for arbitrary L\'evy measures.  In our case, Plancherel's identity yields  
\begin{eqnarray}
\Lambda(f,h)&=&(2\pi)^{-d}\int_{\RR^d}\left\{\int_{\RR^d}\int_0^{\infty}e^{-2t{\psi}(\xi)}|e^{-i\xi\cdot y}-1|^2 \phi(t, y)dt \nu(dy)  \right\} \hat{f}(\xi)\overline{\hat{h}}(\xi)d\xi\\
&=&(2\pi)^{-d}\int_{\RR^d}\left\{2\int_{\RR^d}\int_0^{\infty}e^{-2t{\psi}(\xi)}\left(1-\cos(\xi\cdot y)\right)\phi(t, y)dt \nu(dy)   \right\} \hat{f}(\xi)\overline{\hat{h}}(\xi)d\xi\nonumber\\
&=&(2\pi)^{-d}\int_{\RR^d} m(\xi)\hat{f}(\xi)\overline{\hat{h}}(\xi)d\xi,\nonumber 
\end{eqnarray}
where
\begin{equation}\label{multiplier-2}
m(\xi)=2\int_{\RR^d}\left(1-\cos(\xi\cdot y)\right)\left(\int_0^{\infty}e^{-2t{\psi}(\xi)}\phi(t, y)dt \right)\nu(dy). 
\end{equation}
Thus, $S_{\phi}$ is an $L^p$--Fourier multiplier with $\widehat{S_{\phi}f}(\xi)=m(\xi)\hat{f}(\xi)$, $f\in L^2\cap L^p$, and  $\|m\|_{\infty} \leq \|\phi\|_{\infty}$.  
If $\phi$ is independent of $t$, then  we further get
\begin{eqnarray*}
m(\xi)=\frac{\int_{\RR^d}\left(1- \cos{(\xi \cdot y)}\right)\phi(y) \nu(dy)}{
\int_{\RR^d}\left(1- \cos{(\xi \cdot y)}\right) \nu(dy)},
\end{eqnarray*}
the symbols of \cite{MR2918086}.
Typical examples obtained in this way
are the Marcinkiewicz multipliers \cite{MR2345912} given by 
$$m(\xi_1,\ldots,\xi_d)=\frac{|\xi_j|^{\alpha}}{|\xi_1|^{\alpha}+\ldots+|\xi_d|^{\alpha}},$$
where $0<\alpha<2$ and $j=1, \dots, d$.

Taking $\phi\equiv 1$, the above calculations give 
\begin{corollary}\label{dual} If $f\in L^p(\Rd)\cap L^2(\RR^d)$, $1<p\leq 2$, $h\in L^q(\RR^d)\cap L^2(\RR^d)$, $q=\frac{p}{p-1}$, then 
\begin{equation}
|\int_{\RR^d}f(x)h(x)dx|\leq 2\int_{\RR^d} \widetilde{G}(f)G(h) dx\leq 2\|\widetilde{G}(f)\|_p\|G(h)\|_q. 
\end{equation}
Similarly,  if $f\in L^p(\Rd)\cap L^2(\RR^d)$, $2<p\leq \infty$, $h\in L^q(\RR^d)\cap L^2(\RR^d)$ and $q=\frac{p}{p-1}$, then 
\begin{equation}
|\int_{\RR^d}f(x)h(x)dx|\leq 2\int_{\RR^d} \widetilde{G}(f)G(h) dx\leq 2\|{G}(f)\|_p\|\widetilde{G}(h)\|_q, 
\end{equation}

\end{corollary}

{\bf Acknowledgements.}

Krzysztof Bogdan thanks the 
Department of Statistics of Stanford University for its hospitality during 
this work. We thank Fabrice Baudoin and Krzysztof Michalik for 
discussions on square functions and Jacek Zienkiewicz for a conversation on the proof 
of the Hardy-Stein identity.  We thank Elias Stein for discussions on the results of the paper and on the history of the subject.


\begin{thebibliography}{10}

\bibitem{MR3263924}
David Applebaum and Rodrigo Ba{\~n}uelos.
\newblock Martingale transform and {L}\'evy processes on {L}ie groups.
\newblock {\em Indiana Univ. Math. J.}, 63(4):1109--1138, 2014.

\bibitem{MR2213477}
Dominique Bakry.
\newblock Functional inequalities for {M}arkov semigroups.
\newblock In {\em Probability measures on groups: Recent directions and
  trends}, pages 91--147. Tata Inst. Fund. Res., Mumbai, 2006.

\bibitem{MR855179}
Rodrigo Ba{\~n}uelos.
\newblock Brownian motion and area functions.
\newblock {\em Indiana Univ. Math. J.}, 35(3):643--668, 1986.

\bibitem{MR2928339}
Rodrigo Ba{\~n}uelos.
\newblock The foundational inequalities of {D}. {L}. {B}urkholder and some of
  their ramifications.
\newblock {\em Illinois J. Math.}, 54(3):789--868 (2012), 2010.

\bibitem{MR2918086}
Rodrigo Ba{\~n}uelos, Adam Bielaszewski, and Krzysztof Bogdan.
\newblock Fourier multipliers for non-symmetric {L}\'evy processes.
\newblock In {\em Marcinkiewicz centenary volume}, volume~95 of {\em Banach
  Center Publ.}, pages 9--25. Polish Acad. Sci. Inst. Math., Warsaw, 2011.

\bibitem{MR2345912}
Rodrigo Ba{\~n}uelos and Krzysztof Bogdan.
\newblock L\'evy processes and {F}ourier multipliers.
\newblock {\em J. Funct. Anal.}, 250(1):197--213, 2007.

\bibitem{MR2001941}
Rodrigo Ba{\~n}uelos and Pedro~J. M{\'e}ndez-Hern{\'a}ndez.
\newblock Space-time {B}rownian motion and the {B}eurling-{A}hlfors transform.
\newblock {\em Indiana Univ. Math. J.}, 52(4):981--990, 2003.

\bibitem{MR1707297}
Rodrigo Ba{\~n}uelos and Charles~N. Moore.
\newblock {\em Probabilistic behavior of harmonic functions}, volume 175 of
  {\em Progress in Mathematics}.
\newblock Birkh\"auser Verlag, Basel, 1999.

\bibitem{MR797052}
Andrew~G. Bennett.
\newblock Probabilistic square functions and a priori estimates.
\newblock {\em Trans. Amer. Math. Soc.}, 291(1):159--166, 1985.

\bibitem{2014arXiv1411.7907B}
Krzysztof Bogdan, Yana Butko, and Karol Szczypkowski.
\newblock Majorization, {4G Theorem} and {Schr\"odinger} perturbations.
\newblock {\em ArXiv e-prints}, November 2014.

\bibitem{MR2569321}
Krzysztof Bogdan, Tomasz Byczkowski, Tadeusz Kulczycki, Michal Ryznar, Renming
  Song, and Zoran Vondra{\v{c}}ek.
\newblock {\em Potential analysis of stable processes and its extensions},
  volume 1980 of {\em Lecture Notes in Mathematics}.
\newblock Springer-Verlag, Berlin, 2009.
\newblock Edited by Piotr Graczyk and Andrzej Stos.

\bibitem{MR3251822}
Krzysztof Bogdan, Bart{\l}omiej Dyda, and Tomasz Luks.
\newblock On {H}ardy spaces of local and nonlocal operators.
\newblock {\em Hiroshima Math. J.}, 44(2):193--215, 2014.

\bibitem{MR3339224}
Krzysztof Bogdan, Bart{\l}omiej Siudeja, and Andrzej St{\'o}s.
\newblock Martin kernel for fractional {L}aplacian in narrow cones.
\newblock {\em Potential Anal.}, 42(4):839--859, 2015.

\bibitem{BouLam}
Nicolas Bouleau and Damien Lamberton.
\newblock Th\'eorie de {L}ittlewood-{P}aley-{S}tein et processus stables.
\newblock In {\em S\'eminaire de Probabilit\'es, XX}, volume 1204 of {\em
  Lecture Notes in Mathematics}, pages 162--185. Springer-Verlag, 1986.

\bibitem{MR990239}
Edward~B. Davies.
\newblock {\em Heat kernels and spectral theory}, volume~92 of {\em Cambridge
  Tracts in Mathematics}.
\newblock Cambridge University Press, Cambridge, 1989.

\bibitem{Kar1}
Deniz Karli.
\newblock Harnack inequality and regularity for a product of symmetric stable
  process and {Brownian} motion.
\newblock {\em Potential Anal.}, 38(1):95--117, 2013.

\bibitem{2015arXiv1506.01208}
Daesung Kim.
\newblock Martingale transforms and the {H}ardy-{L}ittlewood-{S}obolev
  inequality for semigroups.
\newblock {\em ArXiv e-prints}, June 2015.

\bibitem{MR2869738}
Ildoo Kim and Kyeong-Hun Kim.
\newblock A generalization of the {L}ittlewood-{P}aley inequality for the
  fractional {L}aplacian $(-\Delta)^{\alpha/2}$.
\newblock {\em J. Math. Anal. Appl.}, 388(1):175--190, 2012.

\bibitem{MR3010850}
Victoria Knopova and Ren{\'e}~L. Schilling.
\newblock A note on the existence of transition probability densities of
  {L}\'evy processes.
\newblock {\em Forum Math.}, 25(1):125--149, 2013.

\bibitem{MR1317805}
Nicolai~V. Krylov.
\newblock A generalization of the {L}ittlewood-{P}aley inequality and some
  other results related to stochastic partial differential equations.
\newblock {\em Ulam Quart.}, 2(4):16--26, 1994.

\bibitem{MR1574750}
John~E. Littlewood and Raymond E. A.~C. Paley.
\newblock Theorems on {F}ourier {S}eries and {P}ower {S}eries, {II}.
\newblock {\em J. London Math. Soc.}, 42(1):57--89, 1937.

\bibitem{Mey2}
Paul-Andr\'e Meyer.
\newblock Demonstration probabiliste de certaines in\'egalit\'es de
  {L}ittlewood-{P}aley.
\newblock In {\em S\'eminaire de Probabilit\'es, X}, volume 511 of {\em Lecture
  Notes in Mathematics}. Springer-Verlag.

\bibitem{Mey1}
Paul-Andr\'e Meyer.
\newblock Retour sur la th\`eorie de {L}ittlewood-{P}aley.
\newblock In {\em S\'eminaire de Probabilit\'es, XV}, volume 850 of {\em
  Lecture Notes in Mathematics}, pages 151--166. Springer-Verlag, 1979.

\bibitem{Mey3}
Paul-Andr\'e Meyer.
\newblock Transformations de {R}iesz pour les lois gaussiennes.
\newblock In {\em S\'eminaire de Probabilit\'es, XVIII}, volume 1059 of {\em
  Lecture Notes in Mathematics}, pages 179--193. Springer-Verlag, 1984.

\bibitem{MR3306689}
Michael Perlmutter.
\newblock On a class of {C}alder\'on-{Z}ygmund operators arising from
  projections of martingale transforms.
\newblock {\em Potential Anal.}, 42(2):383--401, 2015.

\bibitem{MR3185174}
Ken-Iti Sato.
\newblock {\em L\'evy processes and infinitely divisible distributions},
  volume~68 of {\em Cambridge Studies in Advanced Mathematics}.
\newblock Cambridge University Press, Cambridge, 2013.
\newblock Translated from the 1990 Japanese original, Revised edition of the
  1999 English translation.

\bibitem{MR2978140}
Ren{\'e}~L. Schilling, Renming Song, and Zoran Vondra{\v{c}}ek.
\newblock {\em Bernstein functions}, volume~37 of {\em de Gruyter Studies in
  Mathematics}.
\newblock Walter de Gruyter \& Co., Berlin, second edition, 2012.
\newblock Theory and applications.

\bibitem{MR1951521}
Ichiro Shigekawa.
\newblock Littlewood-{P}aley inequality for a diffusion satisfying the
  logarithmic {S}obolev inequality and for the {B}rownian motion on a
  {R}iemannian manifold with boundary.
\newblock {\em Osaka J. Math.}, 39(4):897--930, 2002.

\bibitem{MR0131517}
Elias~M. Stein.
\newblock On the maximal ergodic theorem.
\newblock {\em Proc. Nat. Acad. Sci. U.S.A.}, 47:1894--1897, 1961.

\bibitem{MR0290095}
Elias~M. Stein.
\newblock {\em Singular integrals and differentiability properties of
  functions}.
\newblock Princeton Mathematical Series, No. 30. Princeton University Press,
  Princeton, N.J., 1970.

\bibitem{MR0252961}
Elias~M. Stein.
\newblock {\em Topics in harmonic analysis related to the Littlewood-Paley
  theory}.
\newblock Annals of Mathematics Studies, No. 63. Princeton University Press,
  Princeton, N.J., 1970.

\bibitem{MR583240}
Nicolas~Th. Varopoulos.
\newblock Aspects of probabilistic {L}ittlewood-{P}aley theory.
\newblock {\em J. Funct. Anal.}, 38(1):25--60, 1980.

\bibitem{Zyg}
Antoni Zygmund.
\newblock {\em Trigonometrical Series}.
\newblock Annals of Mathematics Studies, No. 63. Cambridge University Press,
  Cambridge, 1959.

\end{thebibliography}
\end{document}